\documentclass[10pt,a4paper,english]{amsart}
\usepackage{babel}
\usepackage{amsfonts, amsmath, wasysym}
\usepackage{amssymb, amsthm,enumerate}
\usepackage{mathrsfs}
\usepackage{verbatim}
\usepackage[numbers,sort,square]{natbib}
\usepackage{color}
\usepackage[colorinlistoftodos,prependcaption,color=yellow,textsize=tiny]{todonotes}
\newfont{\ffont}{cmr10}

\allowdisplaybreaks

\newcommand{\eps}{\varepsilon}

\newcommand{\xspace}{\hbox{\kern-2.5pt}}

\newcommand{\normm}[1]{{\left\vert\kern-0.25ex\left\vert\kern-0.25ex\left\vert #1 
    \right\vert\kern-0.25ex\right\vert\kern-0.25ex\right\vert}}

\renewcommand{\P}{{\mathbb P}}

\newcommand{\norm}[1]{\left\lVert#1\right\rVert}

\newcommand{\ud}[0]{\,\mathrm{d}}
\newcommand{\dd}[0]{\mathrm{d}}

\begin{document}

\newtheorem{definition}{Definition}
\newtheorem{theorem}[definition]{Theorem}
\newtheorem{conjecture}[definition]{Conjecture}
\newtheorem{proposition}[definition]{Proposition}
\newtheorem{corollary}[definition]{Corollary}
\newtheorem{remark}[definition]{Remark}
\newtheorem{lemma}[definition]{Lemma}
\newtheorem{assumption}[definition]{Assumption}
\newtheorem{example}[definition]{Example}
\newtheorem{exercise}{Exercise}

\title[Tail characterizations of UMD]{Characterizations of the UMD property\\ via tail estimates for tangent processes}

\author{Gergely Bod\'o}
\address{}
\email{g.z.bodo@uva.nl}
\thanks{}

\author{Ivan Yaroslavtsev}
\address{}
\email{yaroslavtsev.i.s@yandex.ru}
\thanks{}


\begin{abstract}
We characterize the UMD property of a Banach space by tail inequalities for maximal functions of tangent conditionally symmetric processes. More precisely, we prove that a Banach space $V$ is UMD if and only if for some (equivalently, for all) $p\in(0,\infty)$ one has that
\[
\mathbb P(\sup_{r\geq 0} \| N_r\|>t)\lesssim_{p,V}\Bigl(\frac{s^p}{t^p}+\mathbb P(\sup_{r\geq 0} \| M_r\|>s)\Bigr),
\qquad s,t>0,
\]
for all tangent conditionally symmetric $V$-valued processes $M$ and $N$. We further show that this estimate is equivalent to suitable Lorentz norm inequalities for the associated maximal functions, and obtain analogous characterizations in the discrete-time, continuous-time, and purely discontinuous settings.
\end{abstract}
\maketitle

\section{Introduction}
The UMD property is one of the central geometric conditions in vector-valued
probability theory and harmonic analysis. Recall that a Banach space $V$ is called UMD if for some (equivalently, for all)
$p \in (1,\infty)$ there exists a constant $\beta>0$ such that
for every $N \geq 1$, every martingale
difference sequence $(d_n)^N_{n=1}$ in $L^p(\Omega; V)$, and every scalar-valued 
sequence
$(\varepsilon_n)^N_{n=1}$ such that $|\varepsilon_n|=1$ for each $n=1,\ldots,N$
we have
\begin{equation}\label{eq:UMD_beta}
\Bigl(\mathbb E \Bigl\| \sum^N_{n=1} \varepsilon_n d_n\Bigr\|^p\Bigr )^{\frac 
1p}
\leq \beta \Bigl(\mathbb E \Bigl \| \sum^N_{n=1}d_n\Bigr\|^p\Bigr )^{\frac 1p}.
\end{equation}
This property is deeply connected to several areas of mathematics including vector-valued singular
integrals, stochastic integration, maximal regularity, and decoupling theory;
see for instance \cite{Burk81,Burk01,Rubio86,HNVW1,Pis16}.

Modern stochastic integration theory in Banach spaces relies heavily on decoupling and related maximal inequalities. This is evident both in the Gaussian UMD framework developed by van Neerven, Veraar, and Weis \cite{NVW07c}, and in Dirksen’s treatment of Poisson integrators in $L^p$-spaces \cite{Dirk14}. More recently, Bod\'o and Riedle \cite{BR24} showed that in Hilbert spaces decoupling methods remain effective even in situations where classical strong $L^p$ bounds are unavailable. In Banach spaces, however, decoupling techniques become rather limited once one moves beyond the regime of processes with finite $p$-th moments for some $p\geq 1$.

The purpose of this work is to extend the decoupling toolkit in Banach spaces by establishing two new characterizations of the UMD property. First, we show that one may characterize the UMD property
by a substantially weaker-looking estimate, which provides a purely distributional
comparison of maximal functions. More precisely, we prove that a Banach space
$V$ is UMD if and only if for any (equivalently, for all) $p\in(0,\infty)$ there exists a constant $\widetilde C_{p,V}>0$ such that for every pair
of $V$-valued tangent conditionally symmetric processes $M$ and $N$ 
one has
\begin{equation}\label{eq:INTRO:L0_to_L0}
\P(N^*>t)\le \widetilde C_{p,V}\Bigl(\frac{s^p}{t^p}+\P(M^*>s)\Bigr), \qquad \text{for all } s,t>0.
\end{equation}
Hence, one does not need a full $L^p$-norm comparison: a suitable control of the tail distribution of $N^*$ in terms of that of $M^*$ already characterizes the UMD property.

Second, we show that this principle extends naturally to the Lorentz scale. More precisely, the UMD property is equivalent to the Lorentz (quasi-)norm estimates
\begin{equation}\label{eq:INTRO:Lorentz_estimates}
  \|N^*\|_{p', q'} \leq C_{p,q,p',q',V} \|M^*\|_{p, q},
\end{equation}
for any (equivalently, for all) $p,p'\in(0,\infty)$ and $q,q'\in(0,\infty]$ such that either $p'=p$ and $q'\ge q$, or $p'<p$, and for some constant $C_{p,q,p',q',V}>0$ depending only on $p$, $q$, $p'$, $q'$, and $V$. In particular, both weak and strong maximal $L^p$ inequalities characterize the UMD property, even for $p<1$. The estimates \eqref{eq:INTRO:L0_to_L0} and \eqref{eq:INTRO:Lorentz_estimates} will be used in our forthcoming work \cite{BYst} on sharp bounds for stochastic integrals with respect to symmetric stable processes and random measures.

Finally, we show that both characterizations continue to hold even if one works only within certain natural subclasses of conditionally symmetric processes, namely continuous local martingales, ``Poisson-like'' purely discontinuous quasi-left continuous processes, or ``discrete-like'' purely discontinuous processes with accessible jumps. Recall that every real-valued local martingale admits a unique decomposition into these three components; moreover, in the vector-valued setting, the availability of such a decomposition characterizes the UMD property, see \cite{Y17MartDec,Y17GMY}. This shows that the mechanism captured by our tail estimates is robust across the main probabilistic frameworks in which decoupling arguments arise.

\medskip

{\em Acknowledgments}: The authors would like to thank Sonja Cox and Stefan Geiss for their helpful suggestions and observations.

\section{Conditionally symmetric processes}

Let $(\Omega,\mathcal F, \mathbb P)$ be a probability space with a filtration $\mathbb F = (\mathcal F_t)_{t\geq 0}$ which satisfies the usual conditions (see \cite{Prot,Kal,JS}).  

 Let $V$ be a Banach space, $M:\mathbb R_+\times \Omega \to V$ be a local martingale. Assume that $M$ has the Meyer-Yoeurp decomposition $M=M^c + M^d$ (i.e.\ the unique decomposition so that $M^c$ is continuous satisfying $M^c_0=0$ and $M^d$ is purely discontinuous; such a decomposition for all local martingales characterizes the UMD property thanks to \cite{Y17GMY}). Then the pair $([\![M^c]\!],\nu^M)$ is called {\em the local characteristics of $M$}, where $[\![M^c]\!]$ is the covariation bilinear form of $M^c$ (i.e.\ $[\![M^c]\!]_t(x^*,y^*)=[\langle M^c, x^*\rangle,\langle M^c, y^*\rangle]_t$ a.s.\ for any $t\geq 0$ and $x^*, y^*\in V^*$), and $\nu^M$ is the compensator of the jump measure $\mu^M$ of $M$ defined 
 by 
 \begin{equation*}\label{eq:defofmuM}
\mu^M([0, t] \times B) := \sum_{0\leq s\leq t} \mathbf 1_{B\setminus\{0\}} (\Delta M_s),\;\;\; t\geq 0, \;\; B \in \mathcal B(V).
\end{equation*}
(We refer the reader to \cite{JS,Yar20} for further information on random measures and local characteristics.)

 Let $M, N:\mathbb R_+\times \Omega \to V$ be local martingales admitting a Meyer-Yoeurp decomposition.  $M$ and $N$ are called {\em tangent} if their local characteristics coincide a.s. $M$ is called {\em conditionally symmetric} if $M$ and $-M$ are tangent (see e.g.\ \cite{JS,Yar20}).

Throughout the note we will be discussing conditionally symmetric {\em processes}, not {\em martingales}. A c\`adl\`ag (a French abbreviation of the phrase ``continuous from right, limits from left'') process  $M:\mathbb R_+\times \Omega \to V$ is called {\em conditionally symmetric} if it admits decomposition into a sum $M=M^c + M^d$ with $M^c$ being a continuous local martingale, if for each $n$ the process
\begin{equation}\label{eq:MndefforMpdqlcnotmart}
t\mapsto M^n_t:= M^c_t + \sum_{0\leq s\leq t} \Delta M_s \mathbf 1_{ 1/n\leq\|\Delta M_s \|\leq n}, \;\;\; t\geq 0,
\end{equation}
is a conditionally symmetric local martingale, and if $M^n\to M$ as $n\to \infty$ in the ucp topology (i.e.\ $(M^n-M)^*_t:=\sup_{0\leq s\leq t}\|M^n_s-M_s\|\to 0$ in probability for any $t\geq 0$ as $n\to \infty$, see \cite[Section II.4]{Prot} for the definition of ucp and further details). A standard symmetric $\alpha$-stable process $L:\mathbb R_+\times \Omega\to \mathbb R$ (as well as any non-zero stochastic integral with respect to such a symmetric $\alpha$-stable process) is not a local martingale provided $\alpha\in (0,1]$ as its compensator is of the form $\nu(\dd x,\ud t)= C|x|^{-\alpha-1}\ud x \ud t$ (see e.g.\ \cite[(14.4)]{Sato}),  so lack of (even local) integrability follows from \cite[Theorem 2.5.2]{AppLPSC} and real-valued decoupling inequalities \cite[Theorem 4.1]{Kal17}, although $L$ is a conditionally symmetric process thanks to the symmetry of the process itself. 

\begin{remark}
In general, one can even construct a L\'evy process $L$ with no finite moments of any order $p>0$. To see this, assume that $L$ has compensator $\nu$ of the form $\nu(\dd x,\dd t)= \rho(x)\ud x\ud t$ with $\int_\mathbb R \rho(x) (x^2 \wedge 1)\ud x<\infty$ but with $\int_{|x|\geq 1} \rho(x)|x|^p\, {\rm d}x=\infty$ for any $p>0$ (see \cite[Theorem 2.5.2]{AppLPSC}). An explicit example of such a $\rho$ is given by $\rho(x)=\mathbf 1_{|x|\geq 2} \frac {1}{|x|\ln(|x|)^2}$ for $x \in \mathbb{R}$. As one readily confirms, for such a L\'evy process, (even local) $L^p$ bounds are not possible for any $p>0$, so the tail estimates of the form \eqref{eq:INTRO:L0_to_L0} are of special interest.
\end{remark}

Two conditionally symmetric processes $M$ and $N$ are called {\em tangent} if $M^n$ and $N^n$ are tangent for any $n\geq 1$. This in particular means that if $M=M^c + M^d$ and $N=N^c+N^d$ are the Meyer-Yoeurp decompositions of $M$ and $N$, then $M^c$ and $N^c$ are tangent as well by \cite[Theorem 3.9]{Yar20} since those are continuous components of $M^n$ and $N^n$ respectively (see e.g.\ \cite{Kal,Y17GMY}).

\begin{remark}
Note that $M$ being conditionally symmetric implies that $\nu^M=\nu^{-M}$ as for any $n\geq 1$ one has 
\[
 \nu^M(\dd x, \ud t) \mathbf 1_{1/n\leq\|x\|\leq n} = \nu^{M^n}(\dd x, \ud t)  =\nu^{-M^n}(\dd x, \ud t) = \nu^{-M}(\dd x, \ud t) \mathbf 1_{1/n\leq\|x\|\leq n}.
\]
However,
not every process satisfying $\nu^M=\nu^{-M}$ can be approximated by local martingales of the form \eqref{eq:MndefforMpdqlcnotmart}, e.g.\ a symmetric $\alpha$-stable process with a nonzero linear drift satisfies $\nu^M=\nu^{-M}$, but $M^n$ would converge to the symmetric non-drift part, not to the original process. Thus one also needs the ucp approximation $M^n\to M$.
\end{remark}

We will later need the following implication of the UMD property (see e.g.\  \cite[Corollary 3.33 and Theorem 5.14]{Yar20}): if $V$ is UMD, then for any $0< p<\infty$ and for any $V$-valued tangent conditionally symmetric martingales $M,N:\mathbb R_+\times \Omega \to V$ we have that
\begin{equation}\label{eq:UMD maximal}
  \mathbb E (N^*)^p
\leq C_{p, V}^p\mathbb E(M^*)^p,
\end{equation}
where $M^* := \sup_{t\geq 0} \|M_t\|$, $N^* := \sup_{t\geq 0} \|N_t\|$, and 
where the minimal admissible constant $C_{p,V}$ is called 
the {\em maximal decoupling UMD$_p$~constant}. Note that $C_{p, V}\geq 1$ as any martingale is tangent to itself. 

The estimate \eqref{eq:UMD maximal} characterizes UMD if $p\geq 1$ (see \cite[Corollary 3.33]{Yar20}). In fact, as we shall see in Corollary \ref{cor:maximal_0<p<1_equiv_UMD} below, \eqref{eq:UMD maximal} characterizes UMD even if $0<p<1$.

\section{Good-$\lambda$ inequalities}
We now start with the following good-$\lambda$ inequalities for general conditionally symmetric processes.

\begin{proposition}\label{prop:goodlambdapdQlc}
Let $V$ be a UMD Banach space, $M$ and $N$ be $V$-valued tangent conditionally symmetric processes. Then for any $p_0\in(0, \infty)$, $\delta> 0$ and any $\beta>\delta + 1$ we have that
\begin{equation}\label{eq:goodlforstochintwwrtrm}
\mathbb P(N^*>\beta\lambda, \Delta M^* \vee\Delta N^* \vee M^* \leq \delta \lambda) \leq \frac{2^{p_0} \delta^{p_0}  C_{p_0,V}^{p_0}}{(\beta-\delta - 1)^{p_0}}\mathbb P(N^*>\lambda),\;\;\; \lambda >0,
\end{equation}
where $\Delta M^* := \sup_{t\geq 0} \|\Delta M_t\|$, $\Delta N^* := \sup_{t\geq 0} \|\Delta N_t\|$, and where $C_{p_0, V}$ is the maximal decoupling UMD$_{p_0}$-constant of $V$ (see \eqref{eq:UMD maximal}).
\end{proposition}

Note that due to \cite[Corollary 3.33]{Yar20} and Doob's maximal inequality one has $C_{p_0, V}\leq \frac{p_0}{(p_0-1)}\beta_{p_0, V}^2$ for $p_0>1$, with $\beta_{p_0,V}$ being the minimal admissible constant $\beta$ in \eqref{eq:UMD_beta} provided $p=p_0$. 

\begin{proof}
Let us first assume that $M$ and $N$ are local martingales. Then the proof follows the lines of the proof of  \cite[Proposition 5.6]{Yar20}, even though $M$ and $N$ there are assumed to be purely discontinuous quasi-left continuous and $p_0\geq 1$, the logic is exactly the same. In particular, introducing the notation of \cite[Proposition 5.6]{Yar20}, we define the stopping times
\begin{equation*}
\begin{split}
\sigma &:= \inf \{ t \geq 0 : \|N_t\| > \lambda \},\\
\tau &:= \inf \{ t \geq 0 : \|M_t\| > \delta \lambda \},\\
\rho &:= \inf \{ t \geq 0 : \|\Delta M_t\| \vee \|\Delta N_t\| > \delta \lambda \},
\end{split}
\end{equation*}
and write $\bar{\mu}^M:=\mu^M-\nu^M$ and $\bar{\mu}^N:=\mu^N-\nu^N$ for compensated jump measures of $M$ and $N$, respectively. With this notation in mind, we define the processes
\begin{equation*}
\begin{split}
\widehat{M}_t&:= \int_{[0,t]\times V}
\mathbf{1}_{\|x\| \le \delta \lambda} \, x \,
\mathbf{1}_{(\tau \wedge \sigma \wedge \rho,\, \tau \wedge \rho]}(s)
\, d\bar{\mu}^M(s,x) +M^c_{\tau \wedge \rho \wedge t} - M^c_{\tau \wedge \sigma \wedge \rho \wedge t}, \quad t \geq 0,\\
\widehat{N}_t &:= \int_{[0,t]\times V}
\mathbf{1}_{\|x\| \le \delta \lambda} \, x \,
\mathbf{1}_{(\tau \wedge \sigma \wedge \rho,\, \tau \wedge \rho]}(s)
\, d\bar{\mu}^N(s,x)+N^c_{\tau \wedge \rho \wedge t} - N^c_{\tau \wedge \sigma \wedge \rho \wedge t}, \quad t \geq 0.
\end{split}
\end{equation*}
 Note that all properties of $\widehat{M}$ and $\widehat{N}$ used in the proof of \cite[Proposition 5.6]{Yar20}, including \cite[Lemma 5.7]{Yar20}, remain valid in the present setting. Indeed, pure discontinuity and quasi-left continuity of $M$ and $N$ play no role in that argument. In particular, one can argue exactly as in \cite[Proposition 5.6]{Yar20} to show that $\widehat{M}$ and $\widehat{N}$ are tangent conditionally symmetric local martingales, so by \eqref{eq:UMD maximal} we have
\[
\mathbb E \sup_{t\geq 0} \|\widehat{N}_t\|^{p_0}
\leq C_{p_0,V}^{p_0} \mathbb E \sup_{t\geq 0} \|\widehat{M}_t\|^{p_0}
\leq C_{p_0,V}^{p_0}2^{p_0}\delta^{p_0}\lambda^{p_0}\mathbb P(N^*>\lambda),
\]
where the latter inequality follows from the fact that $\widehat{M}$ coincides with $M-M^{\tau \wedge \sigma \wedge \rho}$ on $[0,\tau \wedge \rho)$ by construction, so $\|\widehat{M}_t\|\leq 2\delta\lambda$ for all $t\geq 0$, while $\widehat{M}=0$ on $\{\sigma=\infty\}$. Now let
\[
A:=\{N^*>\beta\lambda,\ \Delta M^* \vee \Delta N^* \vee M^* \leq \delta \lambda\}.
\]
On $A$ one has $\tau=\infty$ and $\rho=\infty$ a.s. Moreover, since $\beta>\delta+1>1$, we have $\beta\lambda>\lambda$, and hence $\sigma<\infty$ on $A$. Further, since $\|N_{\sigma-}\|\leq \lambda$ and $\|\Delta N_\sigma\|\leq \delta\lambda$, it follows that $\norm{N_\sigma}\leq (\delta+1)\lambda$. Hence, on $A$, one necessarily has $\widehat{N}^* > (\beta-\delta-1)\lambda$. Therefore, by Markov's inequality,
\[
\mathbb P(A)\leq \mathbb P\bigl(\widehat{N}^*>(\beta-\delta-1)\lambda\bigr)
\leq \frac{\mathbb E (\widehat{N}^*)^{p_0}}{(\beta-\delta-1)^{p_0}\lambda^{p_0}}
\leq \frac{2^{p_0}\delta^{p_0} C_{p_0,V}^{p_0}}{(\beta-\delta-1)^{p_0}}\mathbb P(N^*>\lambda),
\]
which proves \eqref{eq:goodlforstochintwwrtrm} for local martingales.

Now let $M$ and $N$ be general tangent conditionally symmetric processes. Let $M^n$ be defined by \eqref{eq:MndefforMpdqlcnotmart}, $N^n$ be defined analogously. Then $M^n$ and $N^n$ are conditionally symmetric local martingales; moreover, these processes are tangent as $M^c$ and $N^c$ are tangent by \cite[Theorem 3.9]{Yar20} (here $M=M^c + M^d$ and $N=N^c+N^d$ are the Meyer-Yoeurp decompositions of $M$ and $N$), and as
\begin{multline*}
\nu^{M^n}(\dd x, \ud t)  =\nu^{-M^n}(\dd x, \ud t) = \nu^M(\dd x, \ud t) \mathbf 1_{1/n\leq\|x\|\leq n}\\
 =  \nu^N(\dd x, \ud t) \mathbf 1_{1/n\leq\|x\|\leq n}=\nu^{N^n}(\dd x, \ud t)  =\nu^{-N^n}(\dd x, \ud t),
\end{multline*}
consequently, \eqref{eq:goodlforstochintwwrtrm} holds true for $M^n$ and $N^n$. 

Fix \(T\geq 0\). As the argument remains valid on the finite time interval \([0,T]\), and since \(M^n\to M\) and \(N^n\to N\) in ucp, the corresponding truncated suprema converge in probability. Therefore, by taking limits, we may pass from the inequality for \(M^n\) and \(N^n\) to the corresponding inequality for \(M\) and \(N\) on \([0,T]\). Finally, letting \(T\to\infty\), we obtain \eqref{eq:goodlforstochintwwrtrm} by monotone convergence of the underlying events.
\end{proof}

\section{$L^0$ to $L^0$ inequalities}
The first consequence of the previous proposition is the following tail inequality, which in the real-valued setting is known due to de la Pe\~na \cite[Theorem 2.1]{dlP93}, see \cite[Theorem 5.2.2]{KwW92} and \cite[Theorem 6.3.1]{dlPG} for the Hilbert space version. In the vector-valued setting this inequality turns out to characterize the UMD property.

\begin{theorem}\label{thm.decoupling_L_0}
 Let $V$ be a Banach space, $p>0$. Then $V$ is UMD if and only if there exists a constant $\widetilde C_{p, V}$ depending only on $p$ and $V$ such that for any $V$-valued tangent conditionally symmetric processes $M$ and $N$ one has that
 \begin{equation}\label{eq:co.decoupling_L_0}
    \mathbb P(N^*>t)\leq \widetilde C_{p, V} \Bigl(\frac {s^p}{t^p} + \mathbb P(M^*>s)\Bigr),\;\;\; t,s>0.
 \end{equation}
 Moreover, if this is the case, then $ \widetilde C_{p, V}\leq  8^{p\vee 1}C_{p, V}^p$, where $C_{p, V}$ is defined by \eqref{eq:UMD maximal}.
\end{theorem}

\begin{proof}
 {\em The ``only if'' part.} Let $V$ be UMD, $M$, $N$ be as stated above. Fix $t,s>0$. If $t\leq 8s$, then one has that as $C_{p, V}\geq 1$
 \begin{equation}\label{eq:decouplingL0toL0ineq}
8^{p\vee 1}C_{p, V}^p \Bigl(\frac {s^p}{t^p} + \mathbb P(M^*>s)\Bigr) \geq 8^{p\vee 1} \Bigl(\frac 1{8^p} + \mathbb P(M^*>s)\Bigr) \geq 1 \geq  \mathbb P(N^*>t). 
 \end{equation}

Now assume that $t> 8s$.   Fix $p_0> 0$ and $\lambda>0$. By Proposition \ref{prop:goodlambdapdQlc}, we have
 \begin{equation*}
 \begin{split}
    \mathbb P(N^*>\beta\lambda) &=  \mathbb P(N^*>\beta\lambda, \Delta M^* \vee \Delta N^* \vee M^* \leq \delta \lambda)\\
    &\;\;\;\;+  \mathbb P(N^*>\beta\lambda, \Delta M^* \vee\ \Delta N^* \vee M^* > \delta \lambda)\\
    &\leq \eps \mathbb P(N^*>\lambda) + \mathbb P( \Delta M^* \vee \Delta N^* \vee M^* > \delta \lambda)\\
    &\leq \eps \mathbb P(N^*>\lambda) + \mathbb P( \Delta M^*  > \delta \lambda) + \mathbb P( \Delta N^*  > \delta \lambda)+ \mathbb P(  M^* > \delta \lambda)\\
    &\stackrel{(*)}\leq \eps \mathbb P(N^*>\lambda) +7 \mathbb P( \Delta M^*  > \delta \lambda) + \mathbb P(  M^* > \delta \lambda)\\
     &\stackrel{(**)}\leq \eps \mathbb P(N^*>\lambda) +8 \mathbb P( M^*  > \delta \lambda/2) ,
 \end{split}
 \end{equation*}
where $\eps:= \frac{2^{p_0}\delta^{p_0}C_{p_0, V}^{p_0}}{(\beta-\delta-1)^{p_0}}$, $(*)$ follows from \cite[Lemma 5.12]{Yar20} (see also \cite[Lemma 2.3.3]{dlPG}), and $(**)$ holds by $\mathbb P( \Delta M^*  > \delta \lambda) \leq \mathbb P( M^*  > \delta \lambda/2)$ and $\mathbb P( M^*  > \delta \lambda) \leq \mathbb P(  M^*  > \delta \lambda/2)$. Thus, for every $p_0\in(0,\infty)$, $\delta>0$, and $\beta>\delta+1$, one has
\begin{equation}\label{eq:lorentz_goodlambda_start}
\P(N^*/\beta>\lambda)\leq \varepsilon \P(N^*>\lambda)+8\P(2M^*/\delta>\lambda),
\qquad \lambda>0.
\end{equation}

Let us fix $p_0=p$, $\delta=1$, and set $\lambda=2s$ and $\beta=t/\lambda = t/(2s)$, so that \eqref{eq:lorentz_goodlambda_start} yields
 \begin{equation*}
  \begin{split}
  \mathbb P(N^*>t)&\leq \frac{2^p C_{p, V}^p}{(\beta-2)^p} + 8  \mathbb P(M^*>s) = \frac{(4s)^p C_{p, V}^p}{(t-4s)^p} + 8  \mathbb P(M^*>s)\\
  & \stackrel{(*)} \leq \frac{(8s)^p C_{p, V}^p}{t^p} + 8  \mathbb P(M^*>s)  \stackrel{(**)}\leq 8^{p\vee 1}C_{p,V}^p\Bigl(\frac {s^p}{t^p} + \mathbb P(M^*>s)\Bigr),
  \end{split}
 \end{equation*}
 where $(*)$ holds as $t-4s>t/2$ and $(**)$ follows from $C_{p,V}\geq 1$.

  {\em The ``if'' part.} Fix $p\in (0,\infty)$ and assume that \eqref{eq:co.decoupling_L_0} holds for any $V$-valued tangent conditionally symmetric processes $M$ and $N$. In particular,  it holds for Paley-Walsh martingales $f$ and $g$ of the following form:
  \begin{equation}\label{eq:PW_marting_f_g}
  f_k=\sum_{n=1}^k d_n,\;\;\; g_k=\sum_{n=1}^k  v_{n-1}(r_1, \ldots,r_{n-1}) d_n, \;\;\; k\geq 1,
  \end{equation}
where $(r_n)_{n\geq 1}$ is a sequence of independent Rademacher random variables, and $d_n=r_n\phi_{n-1}(r_1, \ldots,r_{n-1})$ for some  $\phi_{n-1}:\{-1,1\}^{n-1}\to  V$, and where $v_{n-1}:\{-1,1\}^{n-1}\to\{-1,1\}$ (see \cite{Burk81}). If in this case $g^*>1$ a.s., then an application of Doob's maximal inequality yields
\begin{equation*}
\mathbb P(g^*>1)\leq \widetilde C_{p, V}(s^p+\mathbb P(f^*>s)) \leq \widetilde C_{p, V}(s^p+\mathbb E \|f_{\infty}\|/s),\;\;\; s>0,
\end{equation*}
consequently, if we set $s=(\mathbb E \|f_{\infty}\|)^{\frac{1}{p+1}}$, then the latter inequality yields 
\begin{equation}\label{eq:f_infty_L1_est_in_terms_of_g*}
\mathbb E \|f_{\infty}\| \geq \left(\frac{\mathbb P(g^*>1)}{2\widetilde C_{p, V}}\right)^{\frac{p+1}{p}},
\end{equation}
and in particular
\begin{equation*}
 g^*> 1\; \text{a.s.}\Longrightarrow\mathbb E \|f_{\infty}\| \geq 1/(2\widetilde C_{p, V})^{\frac{p+1}{p}}.
\end{equation*}
Analogously to the proof of \cite[Theorem 2.1]{Burk81}, the latter inequality leads to $L^p$ estimates of the form \eqref{eq:UMD_beta} for Paley-Walsh $f$ and $g$ defined by \eqref{eq:PW_marting_f_g} for any $p\in(0,\infty)$. These $L^p$ inequalities are equivalent to UMD (see e.g.\  \cite[Theorem 4.2.5]{HNVW1}).
\end{proof}

\begin{remark}\label{re:L0_paley_walsh}
    Notice that it was sufficient to consider Paley-Walsh martingales in the ``if'' part of the previous proof. In particular, it means that Theorem \ref{thm.decoupling_L_0} remains true if we replace ``tangent conditionally symmetric processes'' with ``tangent Paley-Walsh martingales''.
\end{remark}

As Paley-Walsh martingales can be recovered from very different noise families (e.g.\ discrete, continuous, or purely discontinuous), the $L^0$ to $L^0$ decoupling inequalities above happen to characterize the UMD property when working solely with these families. We start with the following discrete version, the proof of which is completely analogous to that of Theorem~\ref{thm.decoupling_L_0}.

\begin{proposition}\label{prop.decoupling_L_0_disc}
  Let $V$ be a Banach space, $p>0$. Then $V$ is UMD if and only if there exists a constant $\widetilde C_{p, V}$ depending only on $p$ and $V$ such that for any $V$-valued tangent conditionally symmetric discrete processes $(f_n)_{n\geq 1}$ and $(g_n)_{n\geq 1}$ one has
 \begin{equation}\label{eq:L0_dec_discr_loc_mart}
  \mathbb P(g^*>t)\leq \widetilde C_{p, V} \Bigl(\frac {s^p}{t^p} + \mathbb P(f^*>s)\Bigr),\;\;\; t,s>0,
 \end{equation}
 where $f^*:= \sup_{n\geq 1}\|f_n\|$, $g^*:= \sup_{n\geq 1}\|g_n\|$.
\end{proposition}

Now we can prove Theorem \ref{thm.decoupling_L_0} for different components of the so-called canonical decomposition, i.e.\ the decomposition of a local martingale into the ``Wiener-like'' continuous part, ``Poisson-like'' quasi-left continuous purely discontinuous part, and ``discrete-like''  purely discontinuous part with accessible jumps (see \cite{Y17GMY,Y17MartDec}). First we show the version for purely discontinuous processes with accessible jumps, where the ``if''  part will be shown to follow from Proposition \ref{prop.decoupling_L_0_disc}. Recall that a conditionally symmetric process $M:\mathbb R_+\times \Omega \to V$ is said to be {\em purely discontinuous} if the continuous part $M^c$ of $M$ is zero, and $M$ is said to have {\em accessible jumps} if $\Delta M_{\tau}=0$ a.s.\ for any finite inaccessible stopping time $\tau$ (i.e.\ for any  stopping time $\tau$ such that $\mathbb P(\tau=\sigma)=0$ for any predictable stopping time $\sigma$, see \cite[Chapter 10]{Kal}).

\begin{corollary}\label{cor:L0_dec_acc_jumps}
 Let $V$ be a Banach space, $p>0$. Then $V$ is UMD if and only if there exists a constant $\widetilde C_{p, V}$ depending only on $p$ and $V$ such that for any $V$-valued tangent purely discontinuous conditionally symmetric processes $M$ and $N$ with accessible jumps one has that
 \begin{equation}\label{eq:cor_L0_dec_acc_jumps}
  \mathbb P(N^*>t)\leq \widetilde C_{p, V} \Bigl(\frac{s^p}{t^p} + \mathbb P(M^*>s)\Bigr),\;\;\; t,s>0.
 \end{equation}
\end{corollary}

\begin{proof}
 UMD implies \eqref{eq:cor_L0_dec_acc_jumps} by Theorem \ref{thm.decoupling_L_0}. The converse is true as any discrete conditionally symmetric process can be presented as a purely discontinuous conditionally symmetric process with accessible jumps in times $\{1,2,3,\ldots\}$, see e.g.\ \cite[Remark 3.2]{Yar20}, so \eqref{eq:cor_L0_dec_acc_jumps} implies \eqref{eq:L0_dec_discr_loc_mart}, which in turn implies UMD by Proposition \ref{prop.decoupling_L_0_disc}.
\end{proof}

Now let us show the version for continuous martingales. Note that continuous conditionally symmetric processes are automatically local martingales due to the definition of  conditionally symmetric processes.

\begin{proposition}\label{prop:L0_dec_ineq_cont_mar}
Let $V$ be a Banach space, $p>0$. Then $V$ is UMD if and only if there exists a constant $\widetilde C_{p, V}$ depending only on $p$ and $V$ such that for any $V$-valued tangent continuous local martingales $M$ and $N$ one has that
 \begin{equation}
\label{eq:prop_L0_dec_ineq_cont_mar}
    \mathbb P(N^*>t)\leq  \widetilde C_{p, V} \Bigl(\frac{s^p}{t^p} + \mathbb P(M^*>s)\Bigr),\;\;\; t,s>0.
 \end{equation}
\end{proposition}

In what follows we will use the following notation:  $x\lesssim_{A} y$ for some $x, y\geq 0$ and some set $A$ if there exists a constant $c_{A}>0$ depending only on $A$ so that $x\leq c_{ A}y$. The notations $\gtrsim_{A}$ and $\eqsim_{A}$ are defined similarly.

\begin{proof}[Proof of Proposition \ref{prop:L0_dec_ineq_cont_mar}]
If $V$ is UMD, then \eqref{eq:prop_L0_dec_ineq_cont_mar} follows from Theorem \ref{thm.decoupling_L_0}.

 Now let \eqref{eq:prop_L0_dec_ineq_cont_mar} hold for any continuous tangent local martingales. Then one can show \eqref{eq:prop_L0_dec_ineq_cont_mar} for Paley-Walsh tangent local martingales as the latter can be approximated in law by stochastic integrals with respect to a Brownian motion (see e.g.\ \cite[Proof of Theorem 3.12]{Y17MartDec}). Indeed, let $W:\mathbb R_+\times \Omega \to \mathbb R$ be a standard Brownian motion over the generated filtration $\mathbb F= (\mathcal F_t)_{t\geq 0}$. Set
 \[
 \tau_0:=0,\;\;\;\tau_n:=\inf\{t\geq \tau_{n-1}:|W_t-W_{\tau_{n-1}}|=1\},\;\;\; n\geq 1.
 \]
Let $f$ and $g$ be Paley-Walsh martingales as in the proof of Theorem~\ref{thm.decoupling_L_0}. As a standard Brownian motion is a strongly Markov process, which additionally almost surely hits the set $\{-1,1\}$, $(W_{\tau_n}-W_{\tau_{n-1}})_{n\geq 1}$ are independent $\{-1,1\}$-valued random variables, and as $W$ is a martingale, these random variables are Rademachers. So without loss of generality we can set $r_n=W_{\tau_n}-W_{\tau_{n-1}}$ and define
\[
M_t:=\sum_{n\geq 1}\int_{\tau_{n-1}\wedge t}^{\tau_n\wedge t}\phi_{n-1}(r_1,\ldots,r_{n-1})\ud W_t,\;\;\; t\geq 0,
\]
\[
N_t:=\sum_{n\geq 1}\int_{\tau_{n-1}\wedge t}^{\tau_n\wedge t}v_{n-1}(r_1,\ldots,r_{n-1})\phi_{n-1}(r_1,\ldots,r_{n-1})\ud W_t,\;\;\; t\geq 0.
\]
First, $M$ and $N$ are continuous local martingales, as the integrands in the $n$-th terms of the sums above are $\mathcal F_{\tau_{n-1}}$-measurable. Second, $M$ and $N$ are tangent since $|v_{n-1}|=1$, so $v_{n-1}$ does not affect the quadratic variation, and thus the local characteristics. Finally, $M_{\tau_n}=f_n$ and $N_{\tau_n}=g_n$ a.s.\ for any $n\geq 1$ as almost surely
\begin{multline*}
\int_{\tau_{n-1}}^{\tau_n}\phi_{n-1}(r_1,\ldots,r_{n-1})\ud W_t\\ = \phi_{n-1}(r_1,\ldots,r_{n-1})(W_{\tau_n}-W_{\tau_{n-1}})
=r_n \phi_{n-1}(r_1,\ldots,r_{n-1})=d_n,
\end{multline*}
similarly for $v_{n-1}(r_1,\ldots,r_{n-1})d_n$; consequently, $g^* = \sup_{n\geq 1}\|N_{\tau_n}\|\leq N^*$, and thus
\begin{equation*}
\begin{aligned}
\mathbb E\|f_{\infty}\| &= \lim_{n\to \infty} \mathbb E\|f_{n}\| = \lim_{n\to \infty} \mathbb E \|M_{\tau_n}\| \stackrel{(i)}=  \lim_{n\to \infty}\lim_{t\to \infty} \mathbb E\|M_{\tau_n\wedge t}\|\\
&\stackrel{(ii)}\gtrsim_{p, V}\liminf_{n\to \infty}\liminf_{t\to \infty}(\mathbb P((N^{\tau_n\wedge t})^*>1))^{\frac{p+1}{p}} = \stackrel{(iii)}\liminf_{n\to \infty}(\mathbb P((N^{\tau_n})^*>1))^{\frac{p+1}{p}}\\
&\stackrel{(iv)}=(\mathbb P(N^*>1))^{\frac{p+1}{p}} \geq(\mathbb P(g^*>1))^{\frac{p+1}{p}},
\end{aligned}
\end{equation*}
where $(i)$ holds by the optional sampling theorem (see e.g.\ \cite[Theorem 1.3.22]{KS}), so $M_{\tau_n\wedge t}=\mathbb E (M_{\tau_n}|\mathcal F_t)$  and by martingale convergence \cite[Theorem 3.3.2]{HNVW1} due to integrability of $M_{\tau_n}$ and the fact that $\tau_n<\infty$ a.s.,  $(ii)$ follows from the argument leading to  \eqref{eq:f_infty_L1_est_in_terms_of_g*},  and $(iii)$ and $(iv)$ hold as both $t\mapsto \mathbb P((N^{\tau_n\wedge t})^*>1)$, $t\geq 0$,  and $(\mathbb P((N^{\tau_n})^*>1))_{n\geq 1}$
are monotonically non-decreasing.  Thus $g^*>1$ a.s.\ implies $\mathbb E\|f_{\infty}\| \gtrsim_{p, V} 1$, which leads to the UMD property similarly to the ``if'' part of Theorem~\ref{thm.decoupling_L_0}.
\end{proof}

Finally, the most complicated but from the authors' point of view the most interesting purely discontinuous quasi-left continuous version. Recall that a conditionally symmetric process $M:\mathbb R_+\times \Omega \to V$ is called quasi-left continuous if $\Delta M_{\tau}=0$ almost surely for every finite predictable stopping time $\tau$ (see e.g.\ \cite[Chapter 10]{Kal}).

\begin{proposition}\label{prop:L0_dec_ipd_qlc_process}
 Let $V$ be a Banach space, $p>0$. Then $V$ is UMD if and only if there exists a constant $\widetilde C_{p, V}$ depending only on $p$ and $V$ such that for any $V$-valued tangent purely discontinuous quasi-left continuous conditionally symmetric processes $M$ and $N$ one has that
 \begin{equation}\label{eq:cor_L0_dec_pd_qlc_process}
  \mathbb P(N^*>t)\leq \widetilde C_{p, V} \Bigl(\frac{s^p}{t^p} + \mathbb P(M^*>s)\Bigr),\;\;\; t,s>0.
 \end{equation}
\end{proposition}

\begin{proof}
If $V$ is UMD, then \eqref{eq:cor_L0_dec_pd_qlc_process} holds by Theorem \ref{thm.decoupling_L_0}.

Conversely, let us assume that \eqref{eq:cor_L0_dec_pd_qlc_process} holds for all $V$-valued tangent purely discontinuous quasi-left continuous conditionally symmetric processes. Our goal is to prove that this implies that the same inequality holds for Paley-Walsh martingales $f$ and $g$ as defined in \eqref{eq:PW_marting_f_g}.

Let $Q,R:\mathbb R_+\times \Omega\to\mathbb R$ be two independent standard Poisson processes, $P:=Q-R$, and let $0<\tau_1<\ldots<\tau_k<\ldots$ be stopping times exhausting all the jumps of $P$ with $\Delta P_{\tau_k}\in\{-1,1\}$ a.s.\ for any $k\geq 1$. Additionally, let $(v_n)$ and $(\phi_n)$ be as in the proof of Theorem \ref{thm.decoupling_L_0}. For any $t\geq 0$ set
\[
M_t:=\sum_{n\geq 1} \Delta P_{\tau_n}\phi_{n-1}( \Delta P_{\tau_1}, \ldots, \Delta P_{\tau_{n-1}})\mathbf 1_{t\geq \tau_n},
\]
and
\[
N_t:=\sum_{n\geq 1} \Delta P_{\tau_n} v_{n-1}( \Delta P_{\tau_1}, \ldots, \Delta P_{\tau_{n-1}})\phi_{n-1}( \Delta P_{\tau_1}, \ldots, \Delta P_{\tau_{n-1}})\mathbf 1_{t\geq \tau_n}.
\]

First, we claim that $M$ and $N$ are martingales. Indeed, for every $n\geq 1$, both $P^{\tau_n}$ and $P^{\tau_{n-1}}$ are martingales by the optional sampling theorem \cite[Theorem 9.12]{Kal}. Hence their difference $t\mapsto P^{\tau_n}_t-P^{\tau_{n-1}}_t=\Delta P_{\tau_n}\mathbf 1_{t\geq \tau_n}$ is also a martingale. Therefore, in order to conclude that $M$ and $N$ are martingales, it is enough to observe that if $\xi$ is a bounded $\mathcal F_{\tau_n}$-measurable random variable such that $t\mapsto \xi\mathbf 1_{t\geq \tau_n}$ is a martingale, and if $\eta$ is a bounded $\mathcal F_{\tau_{n-1}}$-measurable random variable, then $t\mapsto \eta\xi \mathbf 1_{t\geq \tau_n}$ is again a martingale. To see this, let $0\leq s\leq t$. Then
\begin{align*}
\mathbb E (\eta\xi \mathbf 1_{t\geq \tau_n}|\mathcal F_s)
&=\mathbb E (\mathbb E(\eta\xi \mathbf 1_{t\geq \tau_n}|\mathcal F_{s\vee \tau_{n-1}})|\mathcal F_s) \\
&\stackrel{(i)}= \mathbb E \eta (\mathbb E(\xi \mathbf 1_{t\geq \tau_n}|\mathcal F_{s\vee \tau_{n-1}})|\mathcal F_s)\\
&\stackrel{(ii)}= \mathbb E (\eta\xi \mathbf 1_{s\vee \tau_{n-1}\geq \tau_n}|\mathcal F_s)\stackrel{(iii)}=\mathbb E (\eta\xi \mathbf 1_{s\geq \tau_n}|\mathcal F_s)\stackrel{(iv)}=\eta\xi \mathbf 1_{s\geq \tau_n},
\end{align*}
where $(i)$ follows from \cite[Lemma 9.1 (i) and (iii)]{Kal}, $(ii)$ holds since $t\mapsto \xi \mathbf 1_{t\geq \tau_n}$ is a martingale, $(iii)$ follows from $\tau_{n-1}<\tau_n$ a.s., and $(iv)$ follows from \cite[Lemma 9.1 (iii)]{Kal}. Applying this first with $\eta=\phi_{n-1}( \Delta P_{\tau_1}, \ldots, \Delta P_{\tau_{n-1}})$,
and then with
\[
\eta=v_{n-1}( \Delta P_{\tau_1}, \ldots, \Delta P_{\tau_{n-1}})
\phi_{n-1}( \Delta P_{\tau_1}, \ldots, \Delta P_{\tau_{n-1}}),
\]
establishes that all summands in the definitions of $M$ and $N$ are martingales, and hence so are $M$ and $N$.

Second, observe that $M$ and $N$ are purely discontinuous quasi-left continuous. Indeed, this follows from the fact that both are piecewise constant and their jumps occur only at the jump times of $P$, which are totally inaccessible.

Third, $M$ and $N$ are conditionally symmetric and tangent. To see this observe, that if we denote by $\lambda_{I}$ the restriction of the Lebesgue measure to a Borel set $I\subset\mathbb R_+$, and write $\delta_x$ for the Dirac measure at $x\in \mathbb R$, then for each $n\geq 1$, we have almost surely
\begin{equation*}
\begin{aligned}
\nu^{\Delta P_{\tau_n}\mathbf 1_{\cdot\geq \tau_n}}
&=\nu^{P^{\tau_n}-P^{\tau_{n-1}}} =\nu^{P^{\tau_n}-(P^{\tau_{n}})^{\tau_{n-1}}}=\nu^{P^{\tau_n}}-\nu^{(P^{\tau_{n}})^{\tau_{n-1}}}\\
&=\nu^{P}\mathbf 1_{(0, \tau_n]}-\nu^{P}\mathbf 1_{(0, \tau_{n-1}]}
=\nu^{P}\mathbf 1_{(\tau_{n-1}, \tau_n]}
=\lambda_{(\tau_{n-1}, \tau_n]}\otimes(\delta_{-1} + \delta_{1}),
\end{aligned}
\end{equation*}
since $\nu^P=\lambda_{\mathbb R_+}\otimes(\delta_{-1}+\delta_1)$, $\tau_n>\tau_{n-1}$ a.s., and the compensator of a stopped process is a stopped compensator (see e.g.\ \cite[Proposition II.1.30]{JS}). Consequently,
\[
\nu^M
=\sum_{n\geq 1}\lambda_{(\tau_{n-1}, \tau_n]} \otimes (\delta_{-\phi_{n-1}( \Delta P_{\tau_1}, \ldots, \Delta P_{\tau_{n-1}})} + \delta_{\phi_{n-1}( \Delta P_{\tau_1}, \ldots, \Delta P_{\tau_{n-1}})})
\]
almost surely. As the right-hand side of the formula above is symmetric in space and in particular does not change when multiplying  $\phi_{n-1}$ by a $\{-1,1\}$-valued random variable, one can conclude that $\nu^{-M}=\nu^M= \nu^N=\nu^{-N}$, which is equivalent to $M$ and $N$ being tangent conditionally symmetric.

Finally, $(M_{\tau_n})_{n\geq 1}$ and $(N_{\tau_n})_{n\geq 1}$ have the same distributions as $(f_n)_{n\geq 1}$ and $(g_n)_{n\geq 1}$, respectively, since $(\Delta P_{\tau_n})_{n\geq 1}$ is a sequence of independent Rademacher random variables. Indeed, these random variables are $\{-1,1\}$-valued, have zero mean, and $P$ is strongly Markov. Since $M$ and $N$ are constant on each interval $[\tau_{n-1},\tau_n)$, one has
\[
M^*=\sup_{n\geq 1} \|M_{\tau_n}\| \quad \text{and} \quad N^*=\sup_{n\geq 1} \|N_{\tau_n}\| \;\;\text{a.s.}
\]
Hence $M^*$ and $N^*$ have the same distributions as $f^*$ and $g^*$, respectively. Therefore, the estimate \eqref{eq:cor_L0_dec_pd_qlc_process} transfers to the Paley-Walsh martingales $f$ and $g$, and the proof of Theorem \ref{thm.decoupling_L_0} then yields that $V$ is UMD.
\end{proof}

$L^0$ estimates also have the following elementary implications. The first of these establishes that the usual maximal $L^p$ characterization of UMD holds even when $0<p<1$.

\begin{corollary}\label{cor:maximal_0<p<1_equiv_UMD}
The estimate \eqref{eq:UMD maximal} characterizes UMD for any $0<p<\infty$.
\end{corollary}

\begin{proof}
Fix $0<p<\infty$. The fact that UMD implies \eqref{eq:UMD maximal} follows from \cite[Corollary 3.33 and Theorem 5.14]{Yar20}. For the converse, assume that \eqref{eq:UMD maximal} holds. Then as in the proof of Proposition \ref{prop:goodlambdapdQlc}, this yields the good-$\lambda$ inequalities \eqref{eq:goodlforstochintwwrtrm}, which in turn imply the tail estimate \eqref{eq:co.decoupling_L_0}. The latter characterizes UMD by Theorem \ref{thm.decoupling_L_0}.
\end{proof}

We conclude this section with a qualitative characterization of UMD Banach spaces in terms of ucp convergence. This is rather useful in applications, where ucp convergence of a given sequence may be deduced from that of a tangent sequence, which is often easier to analyze.

\begin{corollary}
Let $V$ be a Banach space. Then the following are equivalent:
\begin{enumerate}[\rm(i)]
\item For every pair of sequences $(M^n)_{n \geq 1}$ and $(N^n)_{n \geq 1}$ of tangent conditionally symmetric processes $M^n$ and $N^n$, one has that $M^n \to 0$ in ucp if and only if $N^n \to 0$ in ucp;
\item $V$ is UMD.
\end{enumerate}
\end{corollary}

\begin{proof}
(i) $\Rightarrow$ (ii): Let (i) hold, but $V$ is not UMD. Then thanks to \cite[p.\ 999]{Burk81} there exists a sequence of tangent Paley-Walsh martingales $(f^n)_{n\geq 1}$ and $(g^n)_{n\geq 1}$ such that $(g^n)^*>1$ a.s., but $\|f^n_{\infty}\|_1\to 0$. By a standard procedure, described in \cite[Proof of Theorem 4.21]{Y17GMY}, $(f^n)_{n\geq 1}$ and $(g^n)_{n\geq 1}$ can be distribution-wise represented as tangent continuous-time conditionally symmetric processes $(M^n)_{n \geq 1}$ and $(N^n)_{n \geq 1}$, respectively, over the time interval $[0,1]$. Note that $N^n\nrightarrow 0$ in ucp as $(N^n)^*>1$ a.s., but thanks to Doob's maximal inequality
\[
\mathbb P((M^n)^*_{\infty}>\eps)=\mathbb P((M^n)^*_1>\eps) = \mathbb P((f^n)^*>\eps) \leq \|f^n_{\infty}\|_1/\eps \to 0,\;\;\; n\to \infty,
\]
consequently, $M^n\to 0$ in ucp, which contradicts (i).

(ii) $\Rightarrow$ (i): Now let $V$ be UMD, $(M^n)_{n \geq 1}$ and $(N^n)_{n \geq 1}$ be as described in (i). Assume that $M^n \to 0$ in ucp. For any $T>0$ it follows from \eqref{eq:co.decoupling_L_0}  with $p=1$ and the fact that the stopped processes $(M^n)^T$ and $(N^n)^T$ are tangent for any $n\geq 1$, that for any $t,s>0$
\[
\limsup_{n\to \infty}\mathbb P((N^n)^*_T>t)\lesssim_V \limsup_{n\to \infty} \Bigl(\frac {s}{t} + \mathbb P((M^n)^*_T>s)\Bigr) = \frac {s}{t}.
\]
Let $\eps>0$ be fixed. Then for any $t>0$, we may set $s=\eps t$, so that it follows from the previous estimate that $\limsup_{n\to \infty}\mathbb P((N^n)^*_T>t)\lesssim_V\eps$. Since $\eps>0$ and $T>0$ are arbitrary, we have that $N^n \to 0$ in ucp.
\end{proof}

\section{Lorentz norm estimates}
As a consequence of Theorem \ref{thm.decoupling_L_0} and our good-$\lambda$ inequalities, we show that UMD is equivalent not only to $L^p$ decoupling inequalities,
but can in fact be characterized by any analogous Lorentz-type $L^{p,q}$ decoupling estimate.

For any $p\in(0,\infty)$, $q\in(0,\infty]$, and nonnegative random variable $X$, we
define
\begin{equation}\label{eq:def_of_p,q_norm}
\|X\|_{p,q}
:=
\begin{cases}
\left(\displaystyle\int_0^\infty \bigl(\lambda\, \P(X>\lambda)^{1/p}\bigr)^q\,\frac{d\lambda}{\lambda}\right)^{1/q},
& q\in(0,\infty),\\[2ex]
\displaystyle\sup_{\lambda>0}\lambda\, \P(X>\lambda)^{1/p},
& q=\infty.
\end{cases}
\end{equation}
This is the usual Lorentz quasi-norm written in terms of the distribution function $\P(X>\lambda)$ (see e.g.\ \cite[Proposition 1.4.9]{GrafCl}).
In particular, $L^{p,\infty}$ is the weak-$L^p$ space, and $L^{p,p}=L^p$ up to an
equivalent renorming. 

\begin{remark}\label{rem:emb_Loretnz_spaces}
Note that $\|\cdot\|_{p,q'}\lesssim_{p,q,q'} \|\cdot\|_{p,q}$ by \cite[Proposition 1.4.10]{GrafCl} given $q'\geq q$.
Further, $\|\cdot\|_{p',q'}\lesssim_{p,q,p',q'}\|\cdot\|_{p,q}$ for any $q, q'\in(0,\infty]$ provided $p'< p$ as then 
$$
\|X\|_{p',q'} = \|X\cdot 1\|_{p',q'}\stackrel{(i)}\lesssim_{p,q,p',q'}\|X\|_{p,\infty}\|1\|_{\tfrac{pp'}{p-p'},q'}\stackrel{(ii)}\lesssim_{p,p',q'}\|X\|_{p,\infty}\stackrel{(iii)}\lesssim_{p,q}\|X\|_{p,q},
$$
where $(i)$ follows due to H\"older inequalities for Lorentz spaces (see e.g.\ \cite[Theorem 4.5]{Hunt66}) and the fact that $1/p'=1/p+1/\frac{pp'}{p-p'}$, $(ii)$ holds as in probability spaces $\|1\|_{\tfrac{pp'}{p-p'},q'}<\infty$, and $(iii)$ is a consequence of the first sentence of this remark.
\end{remark}

In the next proof, we will repeatedly use the elementary fact that for any $r\in(0,\infty)$ there exists a
constant $c_r\geq 1$ such that
\begin{equation}\label{eq:elementary_subadditivity_r}
(a+b)^r \leq c_r(a^r+b^r), \qquad a,b\geq 0.
\end{equation}
Indeed, one may take $c_r=1$ if $r\in(0,1]$ and $c_r=2^{r-1}$ if $r> 1$.

\begin{theorem}\label{thm.weak_to_weak_dec}
 Let $V$ be a Banach space, $p,p'\in(0, \infty)$ and $q,q' \in (0,\infty]$ be such that either $p=p'$ and $q'\geq q$, or  $p'<p$.  Then the following are equivalent
\begin{enumerate}[{\rm(i)}]
\item $V$ is UMD,
\item $  \|N^*\|_{p', q'} \lesssim_{p,q, p',q',V} \|M^*\|_{p, q}$ for any $V$-valued tangent conditionally symmetric processes $M$ and $N$.
\end{enumerate}
\end{theorem}

\begin{proof} (i) $\Rightarrow$ (ii):
First assume that $(p',q')=(p,q)$. For a nonnegative random variable $X$, define the shorthands
\[
d_X(\lambda):=\P(X>\lambda)\qquad \text{and }\qquad F_X(\lambda):=\lambda d_X(\lambda)^{1/p}, \qquad \lambda>0.
\]
Then \eqref{eq:lorentz_goodlambda_start} reads
\[
d_{N^*/\beta}(\lambda)\leq \varepsilon d_{N^*}(\lambda)+8d_{2M^*/\delta}(\lambda),
\qquad \lambda>0.
\]
Taking $1/p$-powers and using \eqref{eq:elementary_subadditivity_r} with
$r=1/p$, we obtain
\[
d_{N^*/\beta}(\lambda)^{1/p}
\leq c_{1/p}\bigl(\varepsilon^{1/p}d_{N^*}(\lambda)^{1/p}
+8^{1/p}d_{2M^*/\delta}(\lambda)^{1/p}\bigr),
\qquad \lambda>0.
\]
Multiplying by $\lambda$ yields
\begin{equation}\label{eq:lorentz_F_estimate}
F_{N^*/\beta}(\lambda)
\leq c_{1/p}\varepsilon^{1/p}F_{N^*}(\lambda)
+c_{1/p}8^{1/p}F_{2M^*/\delta}(\lambda),
\qquad \lambda>0.
\end{equation}

Next, note that for any $a>0$, and any nonnegative random variable $X$, we have
\[
d_{aX}(\lambda)=d_X(\lambda/a), \qquad \lambda>0.
\]
Therefore
\[
F_{aX}(\lambda)
=\lambda d_X(\lambda/a)^{1/p}
=aF_X(\lambda/a), \qquad \lambda>0.
\]

Assume first that $q\in(0,\infty)$. Taking the $L^q((0,\infty),d\lambda/\lambda)$
quasi-norm in \eqref{eq:lorentz_F_estimate}, and using again
\eqref{eq:elementary_subadditivity_r}, now with $r=q$, we get
\[
\|F_{N^*/\beta}\|_{L^q((0,\infty),d\lambda/\lambda)}
\leq c_q^{1/q}c_{1/p}\varepsilon^{1/p}\|F_{N^*}\|_{L^q((0,\infty),d\lambda/\lambda)}
\]
\[
\qquad\qquad\qquad\qquad
+c_q^{1/q}c_{1/p}8^{1/p}\|F_{2M^*/\delta}\|_{L^q((0,\infty),d\lambda/\lambda)}.
\]
By the definition of the Lorentz quasi-norm and its homogeneity 
this becomes
\[
\beta^{-1}\|N^*\|_{p,q}
\leq c_q^{1/q}c_{1/p}\varepsilon^{1/p}\|N^*\|_{p,q}
+c_q^{1/q}c_{1/p}8^{1/p}\frac{2}{\delta}\|M^*\|_{p,q}.
\]
Equivalently,
\begin{equation}\label{eq:lorentz_absorb}
\bigl(\beta^{-1}-c_q^{1/q}c_{1/p}\varepsilon^{1/p}\bigr)\|N^*\|_{p,q}
\leq c_q^{1/q}c_{1/p}8^{1/p}\frac{2}{\delta}\|M^*\|_{p,q}.
\end{equation}
Take $p_0>\max\{1,p\}$ and fix $\delta=1$.
Then by the very definition of $\eps$, we have
\[
\varepsilon^{1/p}
=
\frac{2^{p_0/p}C_{p_0,V}^{p_0/p}}{(\beta-2)^{p_0/p}}.
\]
Since $p_0/p>1$, we have $\varepsilon^{1/p}=o(\beta^{-1})$ as $\beta\to\infty$.
Hence, one can choose $\beta>2$ large enough so that
\[
\beta^{-1}-c_q^{1/q}c_{1/p}\varepsilon^{1/p}>0.
\]
Then \eqref{eq:lorentz_absorb} yields
\[
\|N^*\|_{p,q} \leq C_{p,q,V}\|M^*\|_{p,q}.
\]

It remains to consider the case $q=\infty$. Taking the supremum over $\lambda>0$ in
\eqref{eq:lorentz_F_estimate}, we get
\[
\|N^*/\beta\|_{p,\infty}
\leq c_{1/p}\varepsilon^{1/p}\|N^*\|_{p,\infty}
+c_{1/p}8^{1/p}\|2M^*/\delta\|_{p,\infty}.
\]
By homogeneity of the Lorentz quasi-norm,
\[
\bigl(\beta^{-1}-c_{1/p}\varepsilon^{1/p}\bigr)\|N^*\|_{p,\infty}
\leq c_{1/p}8^{1/p}\frac{2}{\delta}\|M^*\|_{p,\infty}.
\]
Choosing again $p_0>\max\{1,p\}$, $\delta=1$, and then $\beta$ large enough, we
obtain
\[
\|N^*\|_{p,\infty} \leq C_{p,\infty,V}\|M^*\|_{p,\infty}.
\]
This completes the proof of the implication given $(p', q')=(p,q)$.   

The case when either $p=p'$ and $q'\geq q$, or  $p'<p$ follows from the inequality for $(p', q')=(p,q)$ and Remark \ref{rem:emb_Loretnz_spaces}.

  (ii) $\Rightarrow$ (i): Let $f$ and $g$ be $V$-valued Paley--Walsh martingales as in \eqref{eq:PW_marting_f_g}. Then $f$ and $g$ are
tangent and conditionally symmetric. Fix $s>0$, and define the stopping time
\[
\tau:=\inf\{n\geq 0: f_n^*>s\; \text{or}\;\|d_{n+1}\|>2s\},
\]
where $\tau$ is a stopping time as $\|d_{n+1}\|=\|\phi_n(r_1,\ldots,r_n)\|$ is $\mathcal F_{n}$-measurable.
Then the stopped discrete-time processes $f^\tau$ and $g^\tau$ are again tangent and
conditionally symmetric (here and later $f_0=g_0=0$, so $f^{\tau}$ and $g^{\tau}$ are well-defined on $\{\tau=0\}$). Hence, by~(ii), and the very definition of $\tau$, we have
\[
\|(g^\tau)^*\|_{p',q'}
\lesssim_{p,q,p',q',V}
\|(f^\tau)^*\|_{p,q},
\]
and $ \|(f^\tau)^*\|_{p,q}\lesssim_{p,q} s$ as for $1\leq n\leq \tau$ one has that a.s.
\[
\|f_n\|=\|f_{n-1}+d_n\|\leq \|f_{n-1}\|+\|d_n\|\leq 3s.
\]
Combining these observations with Remark \ref{rem:emb_Loretnz_spaces}, we obtain
\[
\|(g^\tau)^*\|_{p',\infty}
\lesssim_{p',q'}
\|(g^\tau)^*\|_{p',q'}
\lesssim_{p,q,p',q',V}
s.
\]
Therefore, for every $t>0$, the definition of the $L^{p',\infty}$ quasi-norm gives
\[
\mathbb P((g^\tau)^*>t)
\leq \frac{\|(g^\tau)^*\|_{p',\infty}^{p'}}{t^{p'}}
\lesssim_{p,q,p',q',V}
\frac{s^{p'}}{t^{p'}}.
\]
Now first note that
\[
\{g^*>t,\ \tau=\infty\}\subseteq \{(g^\tau)^*>t\},
\]
from which it immediately follows that
\[
\mathbb P(g^*>t,\ \tau=\infty)
\leq
\mathbb P((g^\tau)^*>t)
\lesssim_{p,q,p',q',V}
\frac{s^{p'}}{t^{p'}}.
\]
Second, if $\tau<\infty$, then either $\|f_{\tau}\|>s$, or $\|f_{\tau}\|<s$ and $\|d_{\tau+1}\|>2s$. In the second case one has 
$$
\|f_{\tau+1}\|=\|f_{\tau} + d_{\tau+1}\|\geq \| d_{\tau+1}\|-\|f_{\tau}\|>s.
$$
Thus $\{\tau<\infty\}\subseteq \{f^*>s\}$, and therefore $\mathbb P(\tau<\infty)\leq \mathbb P(f^*>s)$.

Summing up these two arguments, we obtain the estimate
\begin{equation*}
\begin{aligned}
\mathbb P(g^*>t) &= \mathbb P(g^*>t, \tau<\infty) + \mathbb P(g^*>t, \tau=\infty)\\
&\leq \mathbb P(f^*>s)+\mathbb P(g^*>t,\ \tau=\infty)\\
&\lesssim_{p,q,p',q',V}
\Bigl(\frac{s^{p'}}{t^{p'}}+\mathbb P(f^*>s)\Bigr),
\quad s,t>0,
\end{aligned}
\end{equation*}
which characterizes UMD by Theorem \ref{thm.decoupling_L_0} and Remark \ref{re:L0_paley_walsh}.
  \end{proof}

  \begin{remark}
  As in the latter proof the {\em(ii)} $\Rightarrow$ {\em(i)} direction was shown by only exploiting Paley-Walsh martingales, the characterization of the UMD property provided in Theorem \ref{thm.weak_to_weak_dec} remains in power even if we restrict ourselves to discrete conditionally symmetric processes (similar to Proposition \ref{prop.decoupling_L_0_disc}), to purely discontinuous conditionally symmetric processes with accessible jumps (see Corollary \ref{cor:L0_dec_acc_jumps}), to continuous local martingales (see Proposition \ref{prop:L0_dec_ineq_cont_mar}), or to purely discontinuous quasi-left continuous conditionally symmetric processes (analogously to Proposition \ref{prop:L0_dec_ipd_qlc_process}).
  \end{remark}

\bibliographystyle{plain}
\def\cprime{$'$} \def\polhk#1{\setbox0=\hbox{#1}{\ooalign{\hidewidth
  \lower1.5ex\hbox{`}\hidewidth\crcr\unhbox0}}}
  \def\polhk#1{\setbox0=\hbox{#1}{\ooalign{\hidewidth
  \lower1.5ex\hbox{`}\hidewidth\crcr\unhbox0}}} \def\cprime{$'$}

\end{document}